\documentclass{amsart}
\usepackage{graphicx}
\usepackage{amssymb}
\usepackage{epstopdf}
\usepackage{xcolor}
\usepackage[all,cmtip]{xy}
\usepackage[hidelinks]{hyperref}
\usepackage{enumitem}
\usepackage{mathtools}
\usepackage{etextools}
\usepackage{ifthen}

\newcommand{\Tr}[2]{\operatorname{{\tau}}_{{#2}} (#1)}
\newcommand{\Endo}[2]{\operatorname{End}_{#1} (#2)}

\renewcommand{\hom}[3]{\operatorname{Hom}_{#1} (#2, #3)}

\newcommand{\Ext}[4]{\operatorname{Ext}^{#1}_{#2} (#3, #4)}

\usepackage{mathrsfs}

\newcommand{\cR}{\mathcal R}

\newcommand{\cT}{\mathcal T}
\newcommand{\cU}{\mathcal U}
\newcommand{\cV}{\mathcal V}

\newcommand{\fm}{\frak{m}}

\def\im{\operatorname{Im}}

\def\ker{\operatorname{Ker}}
\def\id{\operatorname{id}}

\def\id{\operatorname{id}}

\def\fm{\mathfrak m}

\def\Hom{\operatorname{Hom}}

\def\Ext{\operatorname{Ext}}

\newcommand{\modR}{\operatorname{mod}R}
\newcommand{\ModR}{\operatorname{Mod}R}
\newcommand{\ModRop}{\operatorname{Mod}R^{\mathrm{op}}}

\newcommand{\Mod}[1]{\mathsf{Mod}(#1)}

\newcommand{\Rop}{R^{\mathrm{op}}}

\newcommand{\trace}[2]{\tau_{#1}(#2)}

\newcommand{\MCM}[1]{\operatorname{MCM}#1}
\newcommand{\fid}[2]{\operatorname{fid}_{#1}#2}
\newcommand{\fpinj}[1]{\operatorname{fpInj}#1}
\newcommand{\Inj}[1]{\operatorname{Inj}#1}

\numberwithin{equation}{section}

\theoremstyle{plain}
\newtheorem{thm}[equation]{Theorem}          \newtheorem*{thm*}{Theorem}
\newtheorem{prp}[equation]{Proposition}      \newtheorem*{prp*}{Proposition}
\newtheorem{cor}[equation]{Corollary}        \newtheorem*{cor*}{Corollary}
\newtheorem{lem}[equation]{Lemma}            \newtheorem*{lem*}{Lemma}
          \newtheorem*{cnj*}{Conjecture}
            \newtheorem*{fct*}{Fact}

\theoremstyle{definition}
\newtheorem{dfn}[equation]{Definition}       \newtheorem*{dfn*}{Definition}
     \newtheorem*{con*}{Construction}
     \newtheorem*{funcon*}{Functorial Constructions}
      \newtheorem*{obs*}{Observation}
\newtheorem{rmk}[equation]{Remark}           \newtheorem*{rmk*}{Remark}
\newtheorem{exa}[equation]{Example}          \newtheorem*{exa*}{Example}
         \newtheorem*{exe*}{Exercise}
         \newtheorem*{qst*}{Question}
            \newtheorem{stp*}{Setup}
            \newtheorem*{set*}{Setting}
            \newtheorem*{ntn*}{Notation}

\title{The trace property in preenveloping classes}

\author[H.\ Lindo]{Haydee Lindo} 
\address{H.L. \ Harvey Mudd College, Claremont, CA 91711}
\email{HLindo@hmc.edu} 

\author[P.\ Thompson]{Peder Thompson} 
\address{P.T. \ Niagara University, NY 14109, and M\"{a}lardalen University, V\"{a}ster{\aa}s, Sweden}
\email{peder.thompson@mdu.se}

\thanks{This material is based upon work supported by the National Science Foundation under Grant No. 2137949}
  \date{March 6, 2023}                                           

\begin{document}

\begin{abstract}
We develop the theory of trace modules up to isomorphism and explore the relationship between preenveloping classes of modules and the property of being a trace module, guided by the question of whether a given module is trace in a given preenvelope. As a consequence we identify new examples of trace ideals and trace modules, and characterize several classes of rings with a focus on the Gorenstein and regular properties.
\end{abstract}

\maketitle
\section{Introduction}
\noindent
Let $R$ be a ring and let $M$ be an $R$-submodule of another $R$-module $X$. The module $M$ is fully invariant in $X$, or endomorphism invariant in $X$, if $\varphi(M) \subseteq M$ for every endomorphism $\varphi$ of $X$. In \cite{JW61}, Johnson and Wong motivated the study of endomorphism invariance by establishing that quasi-injective modules, earlier known for their role in direct sum decompositions of modules, are simply modules that are endomorphism invariant in their injective envelopes; see \cite{Invar,ZimZim}.

This paper explores the relationship between the theory of endomorphism invariant modules, the notion of reflexivity of a module, and the modern theory of trace modules --- particularly in the language of preenveloping classes of modules. Here, an $R$-module $M$ is a trace module if there exist $R$-modules $A$ and $X$ such that 
\[M = \sum_{\alpha \in \Hom_R(A,X)} \alpha(A).\] 

There has been recent interest in the theory of trace ideals over commutative rings, where a trace module in X is called a trace ideal provided $X=R$; see \cite{DKT21, GIK20, canonHTSD, herzog2021set,IsoTrace,Lindo1,PRG21}. This, and other progress in the more general theory of trace modules, have shown certain trace modules to  have desireable properties, like nonrigidity, that is, the existence of self extensions; see \cite{Lin20}. A growing body of work has also demonstrated the ubiquity of trace modules,  and their subsequent usefulness in classifying artinian Gorenstein rings, unique factorization domains, and certain subclasses of hypersurfaces; see also \cite{IsoTrace, LP18}. 

In \cite{BrandtChar}, Brandt's exploration of automorphism invariance in injective modules points to the idea that when $X$ is part of a preenveloping class of $R$-modules then the trace submodules of $X$ are precisely its endomorphism invariant submodules. This paper formalizes this idea and explores its consequences. For example, it follows that quasi-injective modules, and other well-studied classes of endomorphism invariant modules, are examples of trace modules, and the language of trace modules can be used to characterize many more and significantly different classes of rings than previously established. Notably, we observe the following as equivalent to a part of the result \cite[Theorem 6.83]{Lam99} due to work of Osofsky in \cite{Oso64}:
%
$R$ is semisimple if and only if every $R$-module is trace in its injective envelope; see Theorem \ref{SS_char}.
In particular, this is a consequence of our first key technical result about (pre)enveloping classes, Theorem \ref{key_lemma}, which also leads to a characterization of von Neumann regular rings in Corollary \ref{VNR_char}. This complements recent work of Goto, Isobe, and Kumashiro in \cite[Lemma 4.6]{GIK20}, where the authors observe that every ideal in a von Neumann regular ring is a trace ideal.  



The paper is structured as follows. In Section \ref{sec_prelim}, we formalize the idea of both trace submodules and trace modules up to isomorphism, as well as give some of the basic properties of trace modules. In Section \ref{sec_fullyinv}, we explore the link between the trace property and notions such as being fully invariant or reflexive. In particular, these observations provide new examples of trace modules.  We prove two general technical results in Section \ref{sec_trace_preenv}, namely Theorems \ref{key_lemma} and \ref{key_lemma_2},  which provide a way to understand when one class of modules is trace in another class. As consequences, we characterize various rings in terms of traceness, including semisimple rings (Theorem \ref{SS_char}), Gorenstein rings of dimension at most 1 (Theorem \ref{Gor_dim1}), Gorenstein rings of higher dimension (Theorem \ref{Gorinj_char}), von Neumann regular rings (Theorem \ref{VNR_char}), and regular rings (Theorem \ref{char_reg}). Finally, in Section \ref{sec_trace_env}, we consider more closely traceness of ideals in enveloping classes. Our main tool here is Theorem \ref{key2}, from which we obtain a characterization of self-injective rings (Corollary \ref{inj_char}) and another characterization of Gorenstein rings (Theorem \ref{char_Gor}).




\section{Preliminaries on trace modules}\label{sec_prelim}
\noindent 
Throughout this paper, let $R$ be a ring with unity. An $R$-module is assumed to be a left $R$-module; right $R$-modules are considered as modules over the opposite ring, $\Rop$. The category of (left) $R$-modules is denoted $\ModR$. 

We first introduce some terminology. 
\begin{dfn}\label{trace_dfn}
Let $M$ and $X$ be $R$-modules. 
\begin{enumerate}
\item The \emph{trace module of $M$ in $X$} is the sum of all $R$-homomorphic images of $M$ in $X$, and it is denoted $\trace{M}{X}$:
$$\trace{M}{X}=\sum_{\alpha\in \Hom_R(M,X)}\alpha(M).$$
Equivalently, $\trace{M}{X}$ is the image of the natural homomorphism
$$\Hom_R(M,X)\otimes_R M \longrightarrow X$$
defined by evaluation across the tensor, that is, $f\otimes m\mapsto f(m)$ for $f$ in $\Hom_R(M,X)$ and $m$ in $M$. 

\item If $M$ is an $R$-submodule of $X$, we say that $M$ is a \emph{trace submodule} of $X$ if $M=\trace{N}{X}$ for some $R$-module $N$.
In the case $X=R$, a left ideal $I$ of $R$ is a \emph{trace ideal} of $R$ if $I=\trace{N}{R}$ for some $R$-module $N$.
\item If $\iota:M\to X$ is an injective $R$-homomorphism, we say that $M$ is \emph{trace in $X$ via $\iota$} if its image, $\im(\iota)$, is a trace submodule of $X$.
\item More generally, we say that $M$ is \emph{trace in $X$ up to isomorphism} if there exists an injection $\iota:M\to X$ such that $\im(\iota)$ is a trace submodule of $X$.
\end{enumerate}
\end{dfn}

Before moving on, we give a simple example illustrating these notions.
\begin{exa}
 Let $R=k[x]$, for a field $k$.  The ideal $(x)$ is not a trace ideal (that is, not a trace submodule of $R$), however, $(x)$ is {isomorphic} to a trace ideal. Indeed $\trace{(x)}{R}=(1)$, as multiplication by $1/x$ in the total ring of fractions of $R$ gives an isomorphism $\iota:(x)\to (1)$. Using the terminology in Definition \ref{trace_dfn}, we say that $(x)$ is trace in $R$ up to isomorphism, or, more precisely, that $(x)$ is trace in $R$ via $\iota:(x)\to R$.
\end{exa}

There has been recent work by Kobayashi and Takahashi exploring rings in which every ideal is isomorphic to a trace ideal; see \cite{IsoTrace}.

\begin{rmk}
Any trace ideal $I=\trace{N}{R}$ of $R$ is in fact a (two-sided) ideal. This seems to be known, for example see Herbera and P\v{r}\'{\i}hoda \cite{HP14}, but we provide a proof here for completeness.  It is only needed to show the right-absorption property holds: 
    
    Let $x\in I$ and $r\in R$. First suppose that there exists $\varphi:N\to R$ such that $\varphi(y)=x$ for some $y\in N$. Define $\widetilde{\varphi}:N\to R$ by $\widetilde{\varphi}(n)=\varphi(n)r$ for all $n\in N$. The mapping $\widetilde{\varphi}$ is an $R$-homomorphism and so by the definition of $\trace{N}{R}$ one has $xr \in \widetilde{\varphi}(N)\subseteq \trace{N}{R} =  I$.  In general, by the definition of trace modules, if $x \in I = \trace{N}{R}$ then $x$ is a finite sum of the form $x=\sum \varphi_i(y_i)$ for some $\varphi_i:N\to R$ and $y_i\in N$. The previous argument now shows that each $\varphi_i(y_i)r$ belongs to $I$, hence so does $x$.
\end{rmk}

The next fact gives useful alternative descriptions of trace submodules; it is \cite[Lemma 2.4]{Lin20}. Note that \cite{Lin20} subsumes the results from the unpublished \cite{Lindo2}.

\begin{lem}[\cite{Lin20}]\label{tracebasic}
Let $M$ be an $R$-submodule of $X$. The following are equivalent:
\begin{enumerate}
\item[(i)] $M$ is a trace submodule of $X$.
\item[(ii)] $M=\trace{M}{X}$.
\item[(iii)] The inclusion $M\subseteq X$ induces an isomorphism \[\Hom_R(M,M)\xrightarrow{\cong} \Hom_R(M,X).\]
\end{enumerate}
\end{lem}

\begin{rmk}\label{hereditary}
Being a trace submodule of a module is hereditary with respect to submodules:  If $M \subseteq X\subseteq Y$ and $M$ is a trace submodule of $Y$, then $M$ is a trace submodule of $X$. This follows from the equivalence $(i)\Leftrightarrow(iii)$ in Lemma \ref{tracebasic}.
\end{rmk}

\begin{rmk}\label{iso_trace_equal}
Let $M$, $M'$, and $X$ be $R$-modules such that $M\cong M'$. Since $\Hom_R(M,X)\cong \Hom_R(M',X)$, it follows that $\trace{M}{X}=\trace{M'}{X}$.
\end{rmk}
The next fact extends Lemma \ref{tracebasic}.
\begin{lem}\label{tracebasic_iso}
Let $M$ and $X$ be $R$-modules. The following are equivalent:
\begin{enumerate}
    \item[(i)] $M$ is trace in $X$ up to isomorphism.
    \item[(ii)] $M\cong \trace{M}{X}$.
    \item[(iii)] There exists an injective $R$-homomorphism $M\to X$ that induces an isomorphism $\Hom_R(M,M)\xrightarrow{\cong}\Hom_R(M,X)$.
\end{enumerate}
\begin{proof}
Assume $(i)$ holds. There exists an injective $R$-homomorphism $\iota:M\to X$ such that $\im(\iota)$ is a trace submodule of $X$. Now Lemma \ref{tracebasic} and Remark \ref{iso_trace_equal} yield 
$$M\cong \im(\iota)=\trace{\im(\iota)}{X}=\trace{M}{X}\;,$$
and so $(ii)$ holds. Next assuming $(ii)$, there is an evident injective $R$-homomorphism $\iota:M\to X$ induced by the composition $M\cong \trace{M}{X}\subseteq X$. This induces isomorphisms $\Hom_R(M,M)\cong \Hom_R(M,\trace{M}{X})\cong\Hom_R(M,X)$, the latter by Lemma \ref{tracebasic}. Finally, assuming $(iii)$, let $\iota:M\to X$ be an injection such that the induced $\Hom_R(M,M)\to \Hom_R(M,X)$ is an isomorphism. Thus any homomorphism $\varphi:M\to X$ satisfies $\im(\varphi)\subseteq\im(\iota)$, hence $\im(\iota)$ is a trace submodule of $X$ and $(i)$ holds.
\end{proof}

\end{lem}
We next give an example of a trace submodule that is not an ideal.
\begin{exa}
Let $R=k[x]$, for a field $k$.  Consider the $R$-module $X=R/(x^2)$. The $R$-submodule $M=(x)/(x^2)$ is a trace submodule of $X$: indeed $M=\trace{k}{X}$, as for any homomorphism $\varphi:k\to X$, one has $(x)\im(\varphi)=0$, hence $\im(\varphi)\subseteq M$. Thus $k$ is a trace module in $X$ up to isomorphism (via the injection $k\to X$ defined by $1\mapsto x$), but $k$ is not a trace ideal of $R$ (even up to isomorphism).
\end{exa}

\begin{rmk}
Let $M$ be an $R$-module. Given an $R$-homomorphism $X\to Y$, there is an induced map $\trace{M}{X}\to \trace{M}{Y}$, yielding a covariant functor \[\trace{M}{-}:\ModR\to \ModR\]
that is left exact. This can be checked directly from the definition of $\trace{M}{-}$; see also Wisbauer \cite[45.11]{Wis91}. It is not right exact: take $R=k[x]/(x^2)$ and $I=(x)/(x^2)$. Here, $\trace{I}{-}$ does not preserve the surjection $R\to R/I$.
\end{rmk}





\section{Trace, fully invariant, and reflexive modules}\label{sec_fullyinv}
\noindent
In this section we connect the idea of being trace to the classic notions of being fully invariant and being reflexive in an ambient module. This allows us to find new examples of trace modules.

In particular, we next aim to note that trace submodules $M$ in $X$ are stable under endomorphisms of the module $X$, and traces up to isomorphism can be characterized by this fact when $X$ is an envelope from an enveloping class of modules. 

Let $R$ be a ring with unity. The next definition compares to the terminology used in \cite{Lam99}.

\begin{dfn}
Let $\iota:M\to X$ be an injection. We say that $M$ is \emph{fully invariant in $X$ via $\iota$} if each endomorphism $\varphi:X\to X$ restricts to an endomorphism of $\iota(M)$, that is, $\varphi(\iota(M))\subseteq \iota(M)$. If $M$ happens to be an $R$-submodule of $X$ (i.e., $\iota$ is simply inclusion), then we say that $M$ is \emph{fully invariant in $X$}. 
\end{dfn}
In particular, an $R$-submodule $M\subseteq X$ is fully invariant in $X$ provided restriction induces an $R$-algebra homomorphism 
\[ \Endo R X \longrightarrow \Endo R M\;.\]

Recall the following standard definition: 
\begin{dfn}\label{dfn_env}
Let $M$ be an $R$-module and let $\cV$ be a class of $R$-modules.  A \emph{$\cV$-preenvelope} of $M$ is an $R$-homomorphism $\varphi:M\to V$ with $V\in \mathcal{V}$, such that for any $\psi:M\to V'$ with $V'\in \cV$,  there exists $\alpha:V\to V'$ such that $\alpha\varphi=\psi$. A \emph{$\cV$-envelope} of $M$ is a $\cV$-preenvelope $\varphi:M\to V$ such that if $\alpha:V\to V$ satisfies $\alpha\varphi=\varphi$, then $\alpha$ is an isomorphism. Given another class $\cU$ of $R$-modules, we say that $\cV$ is \emph{(pre)enveloping for $\cU$} if every module in $\cU$ has a $\cV$-(pre)envelope.
\end{dfn}
Of course the most well-known example of a $\cV$-(pre)envelope is the case where $\cV=\Inj{R}$, the class of injective $R$-modules. In this case, they are simply referered to as injective (pre)envelopes.

\begin{prp}\label{fullyinv}
Let $\iota:M\to X$ be an injective $R$-homomorphism. If $M$ is trace in $X$ via $\iota$, then it is fully invariant in $X$ via $\iota$. If $M$ is fully invariant in $X$ via $\iota$, and each $R$-homomorphism $M\to X$ lifts to an endomorphism $X\to X$,  for example if $\iota:M\to X$ is a $\cV$-preenvelope for some class $\cV$, 
then $M$ is trace in $X$ via $\iota$.
\end{prp}
\begin{proof}
Assume $M$ is trace in $X$ via $\iota$ and consider any homomorphism $\varphi:X\to X$. Since $M$ is trace in $X$ via $\iota$, the submodule $\iota(M)$ is a trace submodule of $X$. Thus the restriction $\varphi|_{\iota(M)}:\iota(M)\to X$ must have its image contained in $\trace{\iota(M)}{X}=\iota(M)$; see Remark \ref{tracebasic}. This means that $\varphi|_{\iota(M)}$ is an endomorphism of $\iota(M)$ and hence $M$ is fully invariant in $X$ via $\iota$.

Now assume that $M$ is fully invariant in $X$ via $\iota$, that is, $\iota(M)$ is a fully invariant submodule of $X$, and let $\varphi\in \hom R M X$. By assumption, $\varphi$ lifts to a homomorphism $\bar{\varphi} \in \hom R X X$, that is, $\varphi = \bar{\varphi}\iota$. Since $M$ is fully invariant in $X$ via $\iota$,
\[\varphi(M) = \bar{\varphi}\iota(M) \subseteq \iota(M)\;. \] 
Therefore $\iota(M)=\tau_M(X)$. Thus $M$ is trace in $X$ via $\iota$. 
\end{proof}

Recall that as long as a preenveloping class $\cV$ of $\ModR$ contains the injective $R$-modules, then for every $R$-module $M$ the $\cV$-preenvelope $\iota:M\to V$ is an injection.
\begin{cor}\label{trace_fully_inv_envelope}
Let $\iota:M\to V$ be a $\cV$-preenvelope such that $\iota$ is an injection. Then $M$ is trace in $V$ via $\iota$ if and only if $M$ is fully invariant in $V$ via $\iota$.  \qed
\end{cor}

We take a moment to recall a classic notion introduced in \cite{JW61}:
\begin{dfn}
An $R$-module $M$ is called \emph{quasi-injective} if for any submodule $L\subseteq M$, a homomorphism $L\to M$ extends to an endomorphism of $M$.
\end{dfn}
In the case where $\cV$ is the class of injective $R$-modules, the previous corollary says that a module is trace in its injective envelope if and only if it is quasi-injective:
\begin{cor}\label{QI_trace}
Let $M$ be an $R$-module and $M\to E(M)$ its injective envelope. Then $M$ is trace in $E(M)$ via this map if and only if $M$ is quasi-injective.
\end{cor}
\begin{proof}
Let $\iota:M\to E(M)$ be the injective envelope. By \cite[Theorem 6.74]{Lam99}, the submodule $\iota(M)$ is quasi-injective if and only if $\iota(M)$ is fully invariant in $E(M)$.  The result then follows by Corollary \ref{trace_fully_inv_envelope}.
\end{proof}

\begin{exa}
A semisimple $R$-module is trace in its injective envelope. This is because any such module is quasi-injective by \cite[Examples 6.72]{Lam99}.
\end{exa}

In addition to quasi-injectives as a source of trace modules (see \cite[Remark 6.71 and Example 6.72 ]{Lam99}), we have the following result. Compare condition (2) below to the characterization of trace submodules in Lemma \ref{tracebasic}.
\begin{prp}\label{reflexive}
Let $M\subseteq X$ be an $R$-submodule. Suppose the following hold:
\begin{enumerate}
    \item The homothety map $R\to \Hom_R(X,X)$ is a surjection, and 
    \item The induced map $\Hom_R(X,X)\to \Hom_R(M,X)$ is a surjection. 
\end{enumerate}
Then $M$ is a trace submodule of $X$.
\end{prp}
\begin{proof}
Let $\alpha:M\to X$ be a homomorphism. By condition (2), there exists a lifting $\bar{\alpha}:X\to X$ of $\alpha$ along the inclusion $M\subseteq X$. By condition (1), $\bar{\alpha}$ must correspond to multiplication by some element of $R$, thus so does $\alpha$. It follows that $\alpha(M)\subseteq M$, that is, $M$ is a trace submodule of $X$. This idea is based on an argument given in the proof of \cite[Remark 3.2(i)]{LP18}.
\end{proof}


\begin{exa}\label{ex_reflexive}
Here are some examples afforded by Proposition \ref{reflexive}.
\begin{enumerate}
    \item Let $I$ be a left ideal of a ring $R$ such that $\Ext_R^1(R/I,R)=0$. Condition (1) holds trivially and condition (2) holds by the vanishing of Ext applied to the natural long exact sequence in Ext. Thus in this case $I$ is a trace ideal of $R$. In particular, every left ideal of a left self-injective ring is a trace ideal; this is a non-commutative extension of one of the implications of \cite[Theorem 3.5]{LP18}. See also \cite[Example 2.4]{Lindo1}.
    \item Let $R$ be a complete local ring with residue field $k$, and let $E(k)$ be the injective envelope of $k$. Every submodule of $E(k)$ is a trace submodule. Indeed, (1) holds as $R$ is complete by Matlis duality \cite[Theorem 3.2.13]{BH98}, and (2) holds for every submodule of $E(k)$ by injectivity of $E(k)$.
\end{enumerate}
\end{exa}

Goto, Isobe, and Kumashiro \cite[Theorem 4.1]{GIK20} prove that if $R$ is a commutative noetherian ring, then the existence of an embedding of an $R$-module $X$ into $\bigoplus_\fm E_R(R/\fm)$, over all maximal ideals $\fm$ of $R$, is equivalent to the condition that every submodule of $X$ is a trace submodule. The next consequence of Proposition \ref{reflexive} provides some insight into this result over any ring.

\begin{cor}\label{all_submod_trace}
Let $X$ and $Y$ be $R$-modules such that $\Hom_R(Y,Y)\cong R$. Assume that there exists an injective map $\alpha:X\to Y$ such that for every $R$-submodule $M\subseteq X$, the induced map $\alpha^*:\Hom_R(Y,Y)\to \Hom_R(M,Y)$ is surjective. Then every $R$-submodule of $X$ is a trace submodule of $X$.
\end{cor}
\begin{proof}
Let $M\subseteq X$ be an $R$-submodule and suppose $\alpha:X\to Y$ is an injective $R$-homomorphism satisfying the given conditions. By Proposition \ref{reflexive}, $\alpha(M)$ is a trace submodule of $Y$, and hence also of $\alpha(X)$ by Remark \ref{hereditary}. Since $\alpha$ is an injection, $M$ is a trace submodule of $X$.
\end{proof}

We end  the section by discussing $X$-reflexive modules, providing another source of fully invariant and trace submodules.

\begin{dfn}
Let $M$ and $X$ be $R$-modules. We say $M$ is \emph{$X$-reflexive} if the natural map \[
\xymatrix{M\ar[r] & \Hom_R(\Hom_R(M,X),X)\;,}\] given by sending $m$ in $M$ to evaluation at $m$, is an isomorphism.
\end{dfn}

\begin{lem}\label{refl}
Let $M$ and $X$ be $R$-modules such that there exists an injection $\iota: M \longrightarrow X$. 
 If $M$ is $X$-reflexive, 
 then $M$ is fully invariant in $X$ via $\iota$.
\end{lem}
\begin{proof}
It suffices to show that the submodule $\iota(M)\subseteq X$ is fully invariant, thus we may assume $M\subseteq X$.

Application of $\Hom_R(-,X)$ to the exact sequence
\[  \hom R M X \otimes_R M \longrightarrow \trace{M}{X} \longrightarrow 0 \]
yields the first exact row in the following commutative diagram:
\[ \xymatrix{
0 \ar@{->}[r] & \hom R {\Tr M X} X \ar@{->}[r] \ar@{->}[dd]_{\cong}  &  \hom R {\hom R M X \otimes_R M} X  \ar@{->}[d] _{\cong}\\
&&\hom R {M} {\hom R {\hom R M X} {X}} \ar@{->}[d]_{\cong}\\
&\hom R {\Tr X M} {\Tr X M} \ar@{^{(}-->}[r]^{\sigma} & \hom R M M \;,
}
\]
where the isomorphisms in the diagram follow from the traceness of $\Tr X M$ in $X$, Hom-Tensor Adjunction, and the $X$-reflexivity of $M$. Together, these homomorphisms induce a homomorphism $\sigma$ from $\hom R {\Tr X M} {\Tr X M} $ to $ \hom R M M$.

Via the homomorphisms in the above diagram, note that a homomorphism $\psi$ in $\hom R {\Tr X M} {\Tr X M} $ is mapped to the homomorphism $\{m \mapsto \{\alpha \mapsto \psi(\alpha(m)) \} \} $ in $\hom R {M} {\hom R {\hom R M X} {X}}$. Since $M$ is $X$-reflexive, we know that $\{\alpha \mapsto \psi(\alpha(m)) \} = ev_{m'} $ for some $m' \in M$ so that $\psi(\alpha(m) ) = \alpha(m')$ for all $\alpha \in \hom R M X$. In particular, choosing $\alpha$ the inclusion of $M$ in $X$, one has $\psi(\alpha(m) ) = \alpha(m')$ implies $\psi(m) = m'$ and therefore \[\psi(\alpha(m)) = \alpha(m')= \alpha(\psi(m))\] for all $\alpha \in \hom R M X$. We conclude that $\{\alpha \to \psi(\alpha(m)) \}$ is evaluation at $\psi(m)$.  Because $ M \cong  {\hom R {\hom R M X} X}$ via $m \mapsto ev_m$ it follows that $\sigma(\psi) = \psi|_M \in \hom R M M$.



Since $\Tr X M$ is fully invariant in $X$, repeated restriction induces the mapping \[ \hom R X X \longrightarrow \hom R {\Tr X M} {\Tr X M}\longrightarrow \hom R M M\]
and so $M$ is an $\Endo R X$-submodule of $X$. 
\end{proof}
 
\begin{exa}\label{ex_ref}
The following are examples of $R$-modules $M\subseteq X$ such that $M$ is $X$-reflexive, hence fully invariant by Lemma \ref{refl}.

\begin{enumerate}
    \item If $R$ is a commutative Gorenstein ring of dimension 1, then a common example is an ideal containing a nonzerodivisor.  Such an ideal is contained in $R$ and is $R$-reflexive.
    \item If $R$ is a complete local ring with residue field $k$, then any finitely generated submodule of $E(k)$ is $E(k)$-reflexive and hence fully invariant in $E(k)$ by Matlis duality; see \cite[Theorem 3.2.13]{BH98} and compare to Example \ref{ex_reflexive}(2).
    \item  If $R$ is a local Cohen-Macaulay ring, with canonical module $\omega$ and $M$ is a maximal Cohen-Macaulay submodule of $\omega$, then $M$ is $\omega$-reflexive and hence fully invariant in $\omega$; see \cite[Theorem 3.3.10]{BH98}. See also \cite[Corollary 3.3.19]{BH98} for a case where $\omega$ is an ideal.
\end{enumerate}
 \end{exa}


\begin{cor}
Let $\iota:M\to X$ be an injective $R$-homomorphism. If $M$ is both $X$-reflexive and has the lifting property that each homomorphism in $\Hom_R(M,X)$ lifts to one in $\Hom_R(X,X)$, such as when $X$ is a preenvelope, then $M$ is trace in $X$ via $\iota$.  
\end{cor}
\begin{proof}
This follows from Lemma \ref{refl} along with Proposition \ref{fullyinv}.
\end{proof}

\begin{exa}
Let $R$ be a complete local ring with residue field $k$, and let $M$ be a finitely generated submodule of $E(k)$. As noted in Example \ref{ex_ref}(2), $M$ is fully invariant in $E(k)$. Injectivity of $E(k)$ provides the desired lifting property in the previous corollary, hence we actually obtain that $M$ is a trace submodule of $E(k)$.
\end{exa}

\section{Characterizing rings via traceness of subcategories of $\ModR$}\label{sec_trace_preenv}
\noindent
As earlier, $R$ denotes a ring with unity. The aim of this section is to first establish a general result in which we start to answer the question:  given subcategories $\cU$ and $\cV$ of $\ModR$, when is each module in $\cU$ a trace submodule of some module from $\cV$? Our approach is strongly motivated by Corollary \ref{QI_trace} and the proof of \cite[Theorem 6.83]{Lam99}, and as an application we characterize various classes of rings.

This is our first key technical result. It uses the fact that $R$ is a generator for the module category, that is, given any $R$-module $M$ one can find a surjection $R^{(I)}\to M$, for some index set $I$.
\begin{thm}\label{key_lemma}
Let $\cU$ and $\cV$ be subcategories of $\ModR$ such that $\cU$ is closed under finite direct sums and contains $R$, and $\cV$ is closed under isomorphisms and direct summands. Assume further that each module in $\cU$ injects into a module in $\cV$. The following are equivalent:
\begin{itemize}
\item[(i)] The containment $\cU\subseteq \cV$ holds.
\item[(ii)] Every module in $\cU$ is a trace submodule of a module in $\cV$.
\item[(iii)] Every module in $\cU$ is trace in a module in $\cV$ up to isomorphism.
\end{itemize}
If in addition $\cV$ is preenveloping for $\cU$, these are also equivalent to:
\begin{itemize}
\item[(iv)] Every module in $\cU$ is trace in some $\cV$-preenvelope up to isomorphism. 
\end{itemize}
If in addition $\cV$ is enveloping for $\cU$, these are also equivalent to:
\begin{itemize}
\item[$(v)$] Every module in $\cU$ is trace in its $\cV$-envelope up to isomorphism. 
\end{itemize}
\end{thm}
\begin{proof}
The implication $(i)\Rightarrow(ii)$ is clear as every module is a trace submodule of itself, and the implication $(ii)\Rightarrow (iii)$ holds because a trace submodule of an ambient module is trace in the ambient module up to isomorphism (via the inclusion homomorphism).

Consider the implication $(iii)\Rightarrow(i)$: Let $U\in \cU$. By the hypotheses, the sum $U\oplus R$ belongs to $\cU$, and thus by $(iii)$ one has that $U\oplus R$ is trace in some module $V\in \cV$ up to isomorphism. That is, there is an injection $\iota:U\oplus R\to V$ such that $\im(\iota)$ is a trace submodule of $V$. As $R$ is a generator of $\ModR$, so are $U \oplus R$ and  $\im(\iota)$. So, $\im(\iota)$ being a trace submodule of $V$ implies $\im(\iota)=V$, that is, $U\oplus R$ is isomorphic to a module in, and thus belongs to, $\cV$ because $\cV$ is closed under isomorphisms. As $\cV$ is also closed under direct summands, one has that $U\in \cV$. Thus $\cU\subseteq \cV$.

Now suppose that $\cV$ is preenveloping for $\cU$.  The implication $(i)\Rightarrow (iv)$ follows from the fact that the identity map on a module in $\cV$ is a $\cV$-preenvelope. To justify the converse we first show that in this case any $\cV$-preenvelope of a module in $\cU$ is an injection. Let $U\in \cU$ and let $\varphi:U\to V$ be a $\cV$-preenvelope. By hypothesis there is an injection $\iota:U\to V'$ for some $V'\in \cV$. The definition of $\cV$-preenvelope yields a map $\psi:V\to V'$ such that $\psi\varphi=\iota$, and thus $\varphi$ must be an injection. Now, assuming  $(iv)$ let $U\in \cU$. It follows that $U\oplus R$ belongs to $\cU$, and there exists  a $\cV$-preenvelope, $\varphi:U\oplus R \to V$, such that $\varphi(U \oplus R)$ is a trace submodule of $V$. Since $R$ is a generator of $\ModR$, then so are $U\oplus R$ and $\varphi(U\oplus R)$. So $\varphi(U \oplus R)$ being  trace in $V$ implies $\varphi(U \oplus R)=V$.  As $\cV$ is closed under isomorphisms and direct summands, one has that $U\in \cV$, and $(i)$ follows.

The equivalence of $(i)$ and $(v)$ is proven similarly as that of $(i)$ and $(iv)$ mutatis mutandis.
\end{proof}

The equivalence of $(i)$ and $(iii)$ in the following result was first proven in \cite[Theorem 6.83]{Lam99}. As a result of Theorem \ref{key_lemma}  we recover and extend that characterization of semisimple rings:
 
\begin{thm}\label{SS_char} Let $R$ be a ring with unity. The following are equivalent:
\begin{enumerate}
\item[(i)] $R$ is semisimple.
\item[(ii)] Every $R$-module is trace up to isomorphism in its injective envelope.
\item[(iii)] Every $R$-module is quasi-injective.
\end{enumerate}
\end{thm}
\begin{proof}
Applying Theorem \ref{key_lemma} to the subcategories $\cU=\ModR$ and $\cV=\operatorname{Inj}R$, the subcategory of injective $R$-modules, one obtains that every $R$-module is injective (that is, $R$ is semisimple) if and only if every $R$-module is trace up to isomorphism in its injective envelope. The remaining equivalence is by Corollary \ref{QI_trace}.
\end{proof}

The next result is a variation of Theorem \ref{key_lemma}. It will be useful, for example, when considering the map from a module to its double $R$-dual. 
\begin{thm}\label{key_lemma_2}
Let $\cU$ be a subcategory of $\ModR$. Let $F:\ModR\to \ModR$ be an additive functor with a natural transformation $\eta:\id\to F$ such that for each $M\in \cU$ the natural map $\eta_M:M\to F(M)$ is an injection. Assume that $R\in \cU$ and that $\cU$ is closed under finite direct sums. The following are equivalent:
\begin{itemize}
    \item[(i)] For every $M\in\cU$, the map $\eta_M:M\to F(M)$ is an isomorphism.
    \item[(ii)] Every module $M$ in $\cU$ is trace in $F(M)$ via $\eta_M:M\to F(M)$.
\end{itemize}
\end{thm}
\begin{proof}
The implication $(i)\Rightarrow(ii)$ is clear, so it is enough to justify the converse. 
%
Let $M\in \cU$. By assumption, one has $R\in \cU$ and $\cU$ is closed under finite direct sums. Hence $M\oplus R\in \cU$. Thus the natural maps $\eta_{M\oplus R}:M\oplus R\to F(M\oplus R)$ and $\eta_R:R\to F(R)$ are injections. Assumption $(ii)$ implies that $M\oplus R$ is trace in $F(M\oplus R)$ via $\eta_{M\oplus R}$. As $R$ is a generator for $\ModR$, the only way for this to happen is if $\im (\eta_{M\oplus R}) = F(M \oplus R)$, making $\eta_{M\oplus R}$ an isomorphism.

Application of $F$ to the split exact sequence $0 \to M\to M\oplus R \to R \to 0$ yields a commutative diagram with exact rows:
\[\xymatrix{
0 \ar[r] & M \ar[r]\ar@{^(->}[d]_{\eta_M} & M\oplus R \ar[r]\ar[d]_{\eta_{M\oplus R}}^{\cong} & R \ar[r] \ar@{^(->}[d]_{\eta_R} & 0\\
0 \ar[r] & F(M) \ar[r] & F(M\oplus R) \ar[r] & F(R) \ar[r] & 0\;.
}\]
By the Snake Lemma  $\eta_M:M\to F(M)$ is an isomorphism.
\end{proof}

\begin{rmk}
Writing $M^*$ for the  $R$-dual of $M$, $\hom R M R$, there is  a natural \emph{evaluation map}

\begin{center}
\begin{tabular}{cccc}
$\varepsilon_M:$&$M$ & $\longrightarrow $ & $ ({M^*})^*$\\
&$m$& $\mapsto$ & $ev_m: \psi \mapsto \psi(m)\;.$\\
\end{tabular}
\end{center}
\end{rmk}

\begin{dfn} 
A finitely generated $R$-module $M$ is said to be \emph{torsionless} if $\varepsilon_M$ is an injection, and \emph{reflexive} if $\varepsilon_M$ is an isomorphism.  
\end{dfn}

\begin{rmk}\label{TandR}
We will write $\cT$ for the full subcategory of torsionless $R$-modules  and $\cR$ for the full subcategory of reflexive $R$-modules. Note that $\mathcal{R} \subseteq \cT$.

The double dual functor $(-)^{**}$ is additive and there is a natural transformation from the identity functor to it. Moreover, $\cT$ is closed under finite direct sums; see \cite[4.65(c)]{Lam99}.

If $M$ is a module over a commutative noetherian domain, then the double dual $M^{**}$ is a reflexive module; see for example \cite[\href{https://stacks.math.columbia.edu/tag/0AV3}{Tag 0AV3}]{stacks-project}.
\end{rmk}

\begin{lem}\label{reflexive_preenvelope}
Assume that $R$ is a commutative noetherian domain.
The natural homomorphism $M\to M^{**}$ is an $\mathcal{R}$-preenvelope, and $\cR$ is preenveloping in $\ModR$. 
\end{lem}
\begin{proof}
Let $M$ be any $R$-module. There is a natural evaluation homomorphism $\varepsilon_M:M\to M^{**}$. Given any reflexive $R$-module $X$ and homomorphism $\alpha: M \to X$, application of $\Hom_R(\Hom_R(-,R),R)$ induces a homomorphism $\alpha^{**}:M^{**}\to X^{**}$. One may check directly that there is a commutative diagram 
\[\xymatrix{
M \ar[r]^{\varepsilon_M} \ar[d]^\alpha & M^{**}\ar[d]^{\alpha^{**}}\\
X \ar[r]^{\varepsilon_X}_{\cong}  & X^{**}\;.
}\]
The map $\varepsilon_X^{-1}\alpha^{**}:M^{**}\to X$ thus satisfies $\alpha=(\varepsilon_X^{-1}\alpha^{**})\varepsilon_M$, so that $\varepsilon_M:M\to M^{**}$ is an $\mathcal{R}$-preenvelope; see Definition \ref{dfn_env}.
%
%
\end{proof}

\begin{thm}\label{Gor_dim1}
Let $R$ be a commutative noetherian ring. The following are equivalent.
\begin{enumerate}[label=(\roman*)]
\item $R$ is a  Gorenstein ring of dimension at most one. 
\item Every torsionless $R$-module is  reflexive.
\item Every torsionless $R$-module is trace in its double dual, via the natural map. 
\end{enumerate}
If $R$ is also a domain, these conditions are equivalent to:
\begin{enumerate}
    \item[(iv)] Every torsionless $R$-module is trace in some $\cR$-preenvelope up to isomorphism. 
\end{enumerate}
\end{thm}

\begin{proof}
Let $\cT$ be the subcategory of torsionless $R$-modules. We apply Theorem \ref{key_lemma_2} to this subcategory and the double dual functor $(-)^{**}$. Note that for any torsionless $R$-module $M$ the natural map $M\to M^{**}$ is an injection. Also, $\cT$ contains $R$ and is closed under finite direct sums, see Remark \ref{TandR}. The equivalence of $(ii)$ and $(iii)$ thus follows from Theorem \ref{key_lemma_2}. The equivalence of (i) and (ii) is contained in \cite[Theorem 6.2]{Bassubiquity}.

In the case where $R$ is a commutative noetherian domain, apply Theorem \ref{key_lemma} to the subcategories $\cU=\cT$ and $\cV=\cR$. We must check that $\cR$ is closed under direct summands: say $M$ is reflexive and $M=L\oplus N$. Since the double dual functor $(-)^{**}$ is additive, the map $M\to M^{**}$ is simply the sum of maps $L\to L^{**}$ and $N\to N^{**}$, both of which must be isomorphisms as $M\to M^{**}$ is an isomorphism. Now, in light of Lemma \ref{reflexive_preenvelope}, the subcategory $\cR$ is preenveloping and so Theorem \ref{key_lemma} implies that every torsionless $R$-module is reflexive if and only if every torsionless $R$-module is trace in some $\cR$-preenvelope up to isomorphism. This gives the equivalence of $(ii)$ and $(iv$).
\end{proof}

We now turn to characterizing (not-necessarily commutative) Gorenstein rings of higher dimension in terms of trace modules, by utilizing the class of modules having finite Gorenstein injective dimension:
\begin{dfn}
A cochain complex $C^\bullet$ of $R$-modules is called a \emph{totally acyclic complex of injectives} if each $C^i$ is injective and both $C^\bullet$ and $\Hom_R(C^\bullet,J)$ are acyclic for every injective $R$-module $J$. An $R$-module $G$ is called \emph{Gorenstein injective} if there exists a totally acyclic complex of injectives $C^\bullet$ such that $G\cong \ker(C^0\to C^1)$. Finally, the \emph{Gorenstein injective dimension} of an $R$-module $M$ is defined to be the least integer $n$ such that there is an exact sequence $0 \to M \to G^0 \to \cdots \to G^n \to 0$ with each $G^i$ Gorenstein injective. 
\end{dfn}


\begin{dfn}
A ring $R$ is called \emph{Iwanaga-Gorenstein} if $R$ is both left and right noetherian and $R$ has finite injective dimension both as an $R$- and an $\Rop$-module.
\end{dfn}

\begin{thm}\label{Gorinj_char}
Let $R$ be a left and right noetherian ring. The ring $R$ is Iwanaga-Gorenstein of dimension at most $n$ if and only if every  left and right $R$-module is trace up to isomorphism in a module of Gorenstein injective dimension at most $n$. 
\end{thm}
\begin{proof}
Let $\cV$ and $\cV^\mathrm{op}$ be the classes of $R$- and $\Rop$-modules having Gorenstein injective dimension at most $n$.  By \cite[Theorem 12.3.1]{EJ00}, one has that $R$ is an Iwanaga-Gorenstein ring of dimension at most $n$ if and only if both $\ModR\subseteq \cV$ and $\ModRop\subseteq \cV^\mathrm{op}$ hold. In order to apply Theorem \ref{key_lemma}, it is sufficient to check that one can apply this result to the subcategories $\cU=\ModR$ and $\cV$. Note that any $R$-module injects into a module in $\cV$, as $\Inj{R}\subseteq \cV$. Further, $\cV$ is closed under isomorphisms and direct summands. The result follows from Theorem \ref{key_lemma}.
\end{proof}


As recently noted by Goto, Isobe, and Kumashiro \cite[Lemma 4.6]{GIK20}, every ideal in a von Neumann regular ring is a trace ideal. We recall next that such rings can be defined in terms of fp-injective modules, and use this class of modules to characterize von Neumann regular rings in terms of trace modules.

\begin{dfn}
An $R$-module $I$ is called \emph{fp-injective} if $\Ext_R^1(F,I)=0$ for each finitely presented $R$-module $F$. The class of fp-injective $R$-modules is denoted $\fpinj{R}$, and $\fpinj{R}$-preenvelopes are called fp-injective preenvelopes.
\end{dfn}


\begin{rmk}
Von Neumann regular rings are those rings over which all modules are flat, see \cite[Example 2.32(d) and Theorem 4.21]{Lam99}. Equivalently, a von Neumann regular ring is a ring over which all modules are fp-injective; see for example \cite[Remark 1.3]{CT19}. 
\end{rmk}

\begin{thm}\label{VNR_char}
A ring $R$ is von Neumann regular if and only if every $R$-module is trace in an fp-injective preenvelope up to isomorphism.
\end{thm}
\begin{proof}
Apply Theorem \ref{key_lemma} to the categories $\cU=\ModR$ and $\cV=\fpinj{R}$. To do this, recall from \cite[Theorem 3.4]{Trl00} that every $R$-module injects into an fp-injective preenvelope. 
\end{proof}



We end this section with an application to regular rings. Let $(R,\fm,k)$ be a Cohen-Macaulay local ring having a canonical module. If $\dim R=n$, recall (for example from \cite[Definition 11.4]{LW12}) that an $R$-module $\omega$ is a \emph{canonical module} if it is maximal Cohen-Macaulay, has finite injective dimension, and satisfies $\dim_k\Ext_R^n(k,\omega)=1$. 
Let $\modR$ be the subcategory of finitely generated $R$-modules. For an integer $n$, let $\fid{n}{R}\subseteq \modR$ be the subcategory of finitely generated $R$-modules having injective dimension at most $n$ and let $\MCM{R}\subseteq \modR$ be the subcategory of maximal Cohen-Macaulay $R$-modules, that is, those finitely generated $R$-modules having depth equal to $\dim R$.  Note that by \cite[Corollary 5.5]{Bas62}, if $n=\dim R$ one has $\fid{n}{R}$ is precisely the subcategory of finitely generated $R$-modules of finite injective dimension.

The next fact is due to Auslander and Buchweitz \cite{AB89}; see also \cite{LW12}.
\begin{prp}\label{AB_fid_env}
Let $R$ be a Cohen-Macaulay local ring having a canonical module and set $n=\dim R$. The class $\fid{n}{R}$ is enveloping for $\modR$, is closed under direct summands, and every module in $\modR$ injects into a module from $\fid{n}{R}$. 
\end{prp}
\begin{proof}
The fact that $\fid{n}{R}$ is preenveloping follows from \cite[Theorem 11.17]{LW12}; the fact that it is enveloping is proved in a manner dual to \cite[Proposition 11.13]{LW12}. 
It is straightforward to check that $\fid{n}{R}$ is closed under direct summands. The last claim is also by \cite[Theorem 11.17]{LW12}.
\end{proof}


\begin{thm}\label{char_reg}
Let $R$ be a Cohen-Macaulay local ring having a canonical module and set $n=\dim R$. The ring $R$ is regular if and only if every finitely generated $R$-module is trace in its $\fid{n}{R}$-envelope up to isomorphism.
\end{thm}
\begin{proof}
If $R$ is regular, then $\modR=\fid{n}{R}$ and the ``only if'' implication follows. For the converse, we may per Proposition \ref{AB_fid_env} apply Theorem \ref{key_lemma} to the classes of modules $\cU=\MCM{R}$ and $\cV=\fid{n}{R}$. This 
yields that $\MCM{R}\subseteq \fid{n}{R}$.  By \cite[Proposition 11.7]{LW12}, every maximal Cohen-Macaulay $R$-module of finite injective dimension is isomorphic to a direct sum of copies of $\omega$. A high syzygy of the residue field $k$ is MCM, hence also of finite injective dimension, thus contains $\omega$ as a direct summand. 
By \cite[Corollary 4.4]{Tak06} this is enough to imply that $R$ is regular.
\end{proof}

\section{Characterizing rings via traceness of ideals in enveloping classes}\label{sec_trace_env}\noindent
We now turn to characterizing rings $R$ in terms of traceness of ideals in their $\cV$-envelopes, where $\cV$ is an enveloping subcategory of $\ModR$. 
\begin{dfn}
For any class $\cV$ of $R$-modules, one defines the left orthogonal of $\cV$ as the class ${}^\perp \cV=\{M\in \ModR \mid \Ext_R^1(M,V)=0\text{ for all } V\in \cV\}$.
\end{dfn}



\begin{thm}\label{key2}
Let $\cV\subseteq \ModR$ be enveloping.  The following are equivalent:
\begin{itemize}
\item[(i)] The ring $R$ is in $\cV$.
\item[(ii)] Every left ideal $I\subseteq R$ with $R/I\in {}^\perp \cV$ is trace in its $\cV$-envelope up to isomorphism.
\item[(iii)] Every principal left ideal $I\subseteq R$ with $R/I\in {}^\perp \cV$ is trace in its $\cV$-envelope up to isomorphism.
\end{itemize}
Indeed, if $R\in \cV$ and $I$ is a left ideal with $R/I\in {}^\perp \cV$, then $I$ is a trace ideal of $R$.
\end{thm}
\begin{proof}
First assume that $R\in \cV$ and let $I\subseteq R$ be a left ideal with $R/I\in {}^\perp \cV$. One has that $I$ is a trace ideal of $R$ by applying Proposition \ref{reflexive} when $X=R$ and use that the inclusion $I\subseteq R$ induces a surjection $\Hom_R(R,R)\to \Hom_R(I,R)$ because $\Ext_R^1(R/I,R)=0$.

The remainder of the proof is cyclic, where the implication $(ii)\Rightarrow(iii)$ is immediate, and $(iii)\Rightarrow(i)$ follows because $R$ is such a principal ideal, and  because $R$ is a generator in $\Mod R$, $R$ being trace in its $\cV$-envelope up to isomorphism implies that $R$ is in $\cV$, because enveloping classes are closed under isomorphism.

$(i)\Rightarrow(ii)$: Assume that $R\in \cV$ and let $I\subseteq R$ be a left ideal with $R/I\in {}^\perp \cV$. We first show that the $\cV$-envelope of $I$ is isomorphic to a direct summand of $R$. Let $\varphi:I\to V$ be the $\cV$-envelope of $I$. 
As $R\in \cV$ and $\varphi:I\to V$ is a $\cV$-preenvelope, the inclusion $\iota:I\to R$ lifts to a map $\alpha:V\to R$ so that $\alpha\varphi=\iota$. Further, the exact sequence $0\to I \xrightarrow{\iota} R \to R/I \to 0$ induces an exact sequence
\[\xymatrix{
0 \ar[r] & \Hom_R(R/I,V)\ar[r] & \Hom_R(R,V) \ar[r]^{
\iota^*}& \Hom_R(I,V)\ar[r] & 0,
}\]
as $R/I\in {}^\perp \cV$ implies that $\Ext_R^1(R/I,V)=0$. It follows that the map $\varphi:I\to V$ lifts to a map $\beta:R\to V$ such that $\beta\iota=\varphi$. 

Putting this together, we see that $\varphi=\beta\iota=\beta\alpha\varphi$. As $\varphi:I\to V$ is a $\cV$-envelope, it follows that $\beta\alpha$ is an isomorphism, hence $V$ is isomorphic to a direct summand, say $V'$, of $R$. Finally, since $I\subseteq V'\subseteq R$ and we have shown above that $I$ is a trace ideal of $R$, we obtain by Remark \ref{hereditary} that $I$ is a trace submodule of $V'$. Thus $I$ is trace in $V$ up to isomophism.
\end{proof}

As an immediate consequence, we obtain the following:
\begin{cor}\label{inj_char}
The following are equivalent:
\begin{itemize}
\item[(i)] The ring $R$ is injective as a left $R$-module.
\item[(ii)] Every left ideal is trace in its injective envelope up to isomorphism.
\item[(iii)] Every principal left ideal is trace in its injective envelope up to isomorphism.
\end{itemize}
Indeed, if $R$ is injective as a left $R$-module, then every left ideal is a trace ideal.
\end{cor}
\begin{proof}
Apply Theorem \ref{key2} to the class $\cV=\Inj{R}$.
\end{proof}
\begin{rmk}
If $R$ is a commutative noetherian ring, then $R$ is artinian Gorenstein if and only if every ideal is a trace ideal; this is a result of Lindo and Pande \cite[Theorem 3.5]{LP18}, for which the local assumption was removed by Kobayashi and Takahashi \cite[Proposition 3.1]{IsoTrace}. The previous corollary gives the forward implication of this result, and also some insight to the reverse implication for rings that are not necessarily noetherian. For example, Goto, Isobe, and Kumashiro \cite[Lemma 4.6]{GIK20} noted that any commutative von Neumann regular ring has the property that every ideal is a trace ideal, but not all such rings are self-injective, so we do \emph{not} expect a commutative but non-noetherian ring to be self-injective if and only if every ideal is a trace ideal. On the other hand, Corollary \ref{inj_char} says that a ring is injective as an $R$-module if and only if every ideal is trace in its \emph{injective envelope}, and this does not require any commutative or noetherian assumption.
\end{rmk}



We end this section with another application to characterizing Gorenstein rings in the context of a Cohen-Macaulay ring having a canonical module:
\begin{thm}\label{char_Gor}
Let $R$ be a Cohen-Macaulay local ring having a canonical module and set $n=\dim R$. The following are equivalent:
\begin{itemize}
\item[(i)] The ring $R$ is Gorenstein.
\item[(ii)] Every ideal $I$ such that $R/I$ is MCM is trace in its $\fid{n}{R}$-envelope up to isomorphism.
\item[(iii)] Every principal ideal $I$ such that $R/I$ is MCM is trace in its $\fid{n}{R}$-envelope up to isomorphism.
\end{itemize}
\end{thm}
\begin{proof}
The equivalence of (i) -- (iii) follows as an application of Theorem \ref{key2} to the class $\cV=\fid{n}{R}$, which is enveloping by Proposition \ref{AB_fid_env}, and the fact that ${}^\perp \fid{n}{R}=\MCM{R}$; see \cite[Proposition 11.3(i)]{LW12}.  
\end{proof}

\section*{Acknowledgment}
\noindent
Thanks to Steffen Koenig and Henning Krause who invited the first author to visit the Institut f{\"u}r Algebra und Zahlentheorie at Stuttgart University and the Bielefeld Representation Theory Group (BIREP) at Bielefeld University, during which times many of the ideas in this paper were first developed.


\end{document}